\documentclass[12pt]{article} 

\usepackage{amsmath}
\usepackage{amsthm}
\usepackage{amsfonts}
\usepackage{mathrsfs}

\usepackage{setspace}
\usepackage{fullpage}
\usepackage{amssymb}

\usepackage{enumitem}

\usepackage{hyperref}
\usepackage{pgf,tikz}
\usepackage{graphicx}

\newtheorem{thm}{Theorem}%[chapter]

\newtheorem{lemma}[thm]{Lemma}
\newtheorem{conjecture}[thm]{Conjecture}
\newtheorem{proposition}[thm]{Proposition}
\newtheorem{prop}[thm]{Proposition}

\newtheorem{cor}[thm]{Corollary}
\newtheorem{corollary}[thm]{Corollary}

\newtheorem{clm}[thm]{Claim}

\newcommand\ex{\ensuremath{\mathrm{ex}}}

\newcommand\cG{{\mathcal G}}
\newcommand\cH{{\mathcal H}}

\newcommand{\ignore}[1]{}

\title{
	General lemmas for Berge-Tur\'an hypergraph problems
}

\begin{document}
\author{
D\'aniel Gerbner\thanks{Alfr\'ed R\'enyi Institute of Mathematics, Hungarian Academy of Sciences. e-mail: gerbner@renyi.hu}
\qquad
Abhishek Methuku\thanks{ Central European University, Budapest. e-mail: abhishekmethuku@gmail.com}
\qquad
Cory Palmer\thanks{University of Montana, Missoula, Montana 59801, USA. e-mail: cory.palmer@umontana.edu }}

\date{
\today}
\maketitle

\begin{abstract}
For a graph $F$, a hypergraph $\cH$ is a Berge copy of $F$ (or a Berge-$F$ in short), if there is a bijection $f : E(F) \rightarrow  E(\mathcal{H})$ such that for each $e \in E(F)$ we have $e \subset f(e)$. A hypergraph is Berge-$F$-free if it does not contain a Berge copy of $F$. We denote the maximum number of hyperedges in an $n$-vertex $r$-uniform Berge-$F$-free hypergraph by $\ex_r(n,\textrm{Berge-}F).$

In this paper we prove two general lemmas concerning the maximum size of a Berge-$F$-free hypergraph and use them to establish new results and improve several old results. In particular, we give bounds on $\ex_r(n,\textup{Berge-}F)$ when $F$ is a path (reproving a result of Gy\H{o}ri, Katona and Lemons), a cycle (extending a result of F\"uredi and \"Ozkahya), a theta graph (improving a result of He and Tait), or a $K_{2,t}$ (extending a result of Gerbner, Methuku and Vizer). 

We also establish new bounds when $F$ is a clique (which implies extensions of results by Maherani and Shahsiah and by Gy\'arf\'as) and when $F$ is a general tree.
\end{abstract}

\section{Introduction}

Let $F$ be a graph. We say that a hypergraph $\mathcal{H}$ is a \emph{Berge}-$F$ if there is a bijection
$f : E(F) \rightarrow  E(\mathcal{H})$ such that for each $e \in E(F)$ we have $e \subset f(e)$. 
In other words, $\mathcal{H}$ is a Berge-$F$ if we can embed a distinct graph edge into each hyperedge of $\mathcal{H}$ to create a copy of the graph $F$ on the vertex set of $\mathcal{H}$. 
Viewed another way, given a graph $F$ we can construct a Berge-$F$ by replacing each edge of $F$ with a hyperedge that contains it. 
Observe that for a fixed $F$ there are many hypergraphs that are a Berge-$F$. For simplicity, we use the term ``Berge-$F$'' to refer to this collection of hypergraphs. This definition was introduced by Gerbner and Palmer \cite{gp1} to generalize the established concepts of ``Berge path'' and ``Berge cycle'' to general graphs.

\vspace{2mm}

For a fixed graph $F$, if a hypergraph $\mathcal{H}$ has no subhypergraph isomorphic to any Berge-$F$ we say that $\mathcal{H}$ is \emph{Berge-$F$-free}. We denote the maximum number of hyperedges in an $n$-vertex $r$-uniform Berge-$F$-free hypergraph by 
\[
\ex_r(n,\textrm{Berge-}F).
\]

In this paper we prove two general lemmas and use them to prove several new results and give new short proofs of previously-known bounds on $\ex_r(n,\textrm{Berge-}F)$ for various graphs $F$. In most of these cases, these proofs lead to improved theorems.

\vspace{2mm}

Results of Gy\H ori, Katona, and Lemons \cite{GyKaLe} and Davoodi, Gy\H ori, Methuku and Tompkins \cite{DavoodiGMT} establish an analogue of the Erd\H os-Gallai theorem for Berge paths.
Gy\H{o}ri and Lemons \cite{Gyori_Lemons} proved that the maximum number of hypereges in an $n$-vertex $r$-uniform $\textup{Berge-}C_{2k}$-free hypergraph (for $r \ge 3$) is $O(n^{1+1/k})$. This matches the order of magnitude of the bound found in the graph case (see the even cycle theorem of Bondy and Simonovits \cite{BS1974}). They also prove the unexpected result that the maximum number of hyperedges in an $n$-vertex $r$-uniform $\textup{Berge-}C_{2k+1}$-free hypergraph (for $r \ge 3$) is also $O(n^{1+1/k})$ which is significantly different from the graph case. Very recently, the problem of avoiding all Berge cycles of length at least $k$ has been investigated in a series of papers \cite{fkl,EGMNSTLongBerge,KostochkaLuo}.
For general results on the maximum size of a Berge-$F$-free hypergraph for an arbitrary graph $F$ see Gerbner and Palmer \cite{gp1} and Gr\'osz, Methuku and Tompkins \cite{GMTthreshold}.

\vspace{2mm}

A notion related to Berge-Tur\'an problems is the subgraph-counting problem. Following Alon and Shikhelman \cite{AlonS}, let us denote the maximum number of copies of a graph $H$ in an $n$-vertex $F$-free graph by
\[
\ex(n,H,F).
\]

For an overview of bounds on $\ex(n,H,F)$ see \cite{AlonS, gp2}. % how much history do we want to say here?
A simple observation (that will be helpful later) by Gerbner and Palmer \cite{gp2} connects these two areas,

\begin{prop}[Gerbner, Palmer \cite{gp2}]\label{Bergecontainment}

For any graph $F$ we have \[\ex(n,K_r,F)\le \ex_r(n,\textup{Berge-}F) \le \ex(n,K_r,F)+\ex(n,F).\]

\end{prop}

Note that this proposition yields asymptotics for many graphs $F$ as $\ex(n,F)$ gives at most a quadratic gap between the upper and lower bounds.

\vspace{2mm}

Another connection appears in a paper of Palmer, Tait, Timmons and Wagner \cite{pttw} which follows from the combination of results from Mubayi and Verstra\"ete \cite{MuVe} and Alon and Shikhelman \cite{AlonS}. In particular, when $F$ is a graph with chromatic number $\chi(F) > r$, then
\[
\ex(n,K_r,F) \sim \ex_r(n,\textrm{Berge-}F) \sim \binom{k-1}{r}\left(\frac{n}{k-1}\right)^r.
\]

In the next section we prove two main lemmas that will be used in the later sections to derive new results and improve several existing theorems.
In particular, we give bounds on $\ex_r(n,\textup{Berge-}F)$ when $F$ is a path (reproving a result of Gy\H{o}ri, Katona and Lemons), a cycle (extending a result of F\"uredi and \"Ozkahya), a theta graph (improving a result of He and Tait), or a $K_{2,t}$ (extending a result of Gerbner, Methuku and Vizer). 

We also establish new bounds when $F$ is a clique (which implies extensions of results by Maherani and Shahsiah and by Gy\'arf\'as) and when $F$ is a general tree.

\section{General lemmas for Berge-\texorpdfstring{$F$}{F}-free hypergraphs}

Throughout this section we will deal with graphs that are $2$-edge-colored with colors red and blue. 
We will refer to a graph with such an edge-coloring as a {\it red-blue graph}.
In a red-blue graph $G$ we denote the graph spanned by the red edges by $G_{\textrm{red}}$ and the graph spanned by the blue edges by $G_{\textrm{blue}}$. Let $\mathcal{N}(H,G)$ denote the number of copies of the graph $H$ in the graph $G$. Finally, for a positive integer $r$ and a red-blue graph $G$, put
\[
g_r(G)=e(G_{\textrm{red}}) + \mathcal{N}(K_r,G_{\textrm{blue}}).
\]

For several graphs $F$, our theorems in following sections imply that the lower bound in Proposition~\ref{Bergecontainment} is sharp. However, one can easily see that the lower bound is not always sharp. For example, if $r>|V(F)|$, an $F$-free graph cannot contain a clique of size $r$, but a Berge-$F$-free hypergraph can contain hyperedges of size $r$. 

The following lemma improves the upper bound in Proposition~\ref{Bergecontainment}.
Note that an essentially equivalent statement (with different proof) has recently been independently discovered by F\"uredi, Kostochka and Luo (\cite{fkl}, Lemma 4.3).

\begin{lemma}\label{main}
	Let $\mathcal{H}$ be an $r$-uniform Berge-$F$-free hypergraph. Then we can construct a $F$-free red-blue graph $G$ 
    such that
	\[|\mathcal{H}| \leq g_r(G) = e(G_{\textrm{red}}) + \mathcal{N}(K_r,G_{\textrm{blue}}).\]
\end{lemma}

\begin{proof} 
Let $\mathcal{H}$ be an $n$-vertex $r$-uniform Berge-$F$-free hypergraph with $\ex_r(n,\textup{Berge-}F)$ hyperedges. Define an auxiliary bipartite graph $X$ with classes $A$ and $B$ as follows. The class $A$ is the set of hyperedges of $\mathcal{H}$ and the class $B$ is the set of all pairs of vertices contained in some hyperedge of $\mathcal{H}$. 
A vertex $b \in B$ is joined to $a \in A$ if the pair of vertices corresponding to $b$ is contained in the hyperedge corresponding to $a$.
A matching in $X$ induces a set of vertices in $B$; these vertices correspond to a set of pairs on the vertex set of $\mathcal{H}$, i.e., they form a graph on the vertex set of $\mathcal{H}$.

Let $M$ be a maximum-size matching in $X$ and let $G$ be the graph corresponding to the endpoints of $M$ in $B$. Note that $G$ is $F$-free as each edge of $G$ is associated with a hyperedge in $\mathcal{H}$ that contains it.
If $M$ saturates $A$, then $G$ satisfies the statement of the theorem if we color all edges red.

Now consider the case when $M$ does not saturate $A$. An {\em alternating path} in $X$ is a path that alternates between edges in $M$ and edges not in $M$ (beginning with an edge of $M$).
Let $A_1 \subset A$ and $B_1 \subset B$ be the vertices of $X$ that are not in $M$. As $M$ is maximum, there are no edges between $A_1$ and $B_1$.
Let $A_2 \subset A$ be vertices of $M$ in $A$ that are connected by an alternating path to a vertex in $B_1$. Let $B_2 \subset B$ be vertices matched to $A_2$ by $M$. 
Suppose there is an edge $ab$ where $a \in A \setminus A_2$ and $b \in B_2$. By definition, there is an alternating path from $a$ to a vertex in $A_1$. Adding the edge $ab$ to this alternating path gives an alternating path with both start and end edges not in $M$, i.e., $M$ would not be maximal. Therefore, every edge incident to $B_2$ is incident to $A_2$.

Similarly, let $B_3 \subset B$ be vertices of $M$ in $B$ that are connected by an alternating path to a vertex in $A_1$. Let $A_3 \subset A$ be the vertices matched to $B_3$ by $M$.
For any edge $ab$ of $M$, if there is an alternating path from $a$ to a vertex in $B_1$, then there is no alternating path from $b$ to a vertex in $A_1$ otherwise $M$ can be increased. Therefore, $A_2$ and $A_3$ are disjoint as are $B_2$ and $B_3$.

Finally, let $A_4$ and $B_4$ be the remaining vertices in $A$ and $B$, respectively.
 Thus $G$ is spanned by the pairs represented by vertices in $B_2 \cup B_3 \cup B_4$. 
 
Now color red the edges of $G$ that are represented by vertices of $B_2$ and color blue the edges of $G$ represented by vertices of $B_3 \cup B_4$.
The number of hyperedges in $\mathcal{H}$ is
\[
|\mathcal{H}| = |A_1|+|A_2|+|A_3|+|A_4| = |B_2|+|A_1|+|A_3|+|A_4| = e(G_{\textup{red}}) + |A_1|+|A_3|+|A_4|.
\]
The vertices in $A_1 \cup A_3 \cup A_4$ are only adjacent to vertices in $B_3 \cup B_4$. Thus 
\[|A_1|+|A_3|+|A_4| \leq \mathcal{N}(K_r,G_{\textrm{blue}}).\]
Thus $|\mathcal{H}| \leq e(G_{\textrm{red}}) + \mathcal{N}(K_r,G_{\textrm{blue}})$.
\end{proof}

Observe
that Lemma~\ref{main} implies the following corollary
\begin{cor}\label{blue-redtobergeF}
\[
\ex_r(n,\textup{Berge-}F) \le \max \{ g_r(G): G \textrm{ is an $n$-vertex $F$-free red-blue graph}\}.
\]
\end{cor}

Now we prove a general lemma that will be used throughout the later sections.

\begin{lemma}\label{genenew} 
Let $F$ be a graph and let $F'$ be a graph resulting from the deletion of a vertex from $F$. Let $c = c(n)$ be such that $\ex(n, K_{r-1},F')\le c n$ for every $n$. Then
 $$\ex_r(n,\textup{Berge-}F)\le \max\left\{\frac{2c}{r}, 1 \right\} \ex(n,F).$$
\end{lemma}

\begin{proof}

Let $G$ be an $n$-vertex $F$-free red-blue graph.
Let $m$ be the number of blue edges in $G$. Thus, the number of red edges in $G$ is at most $\ex(n,F)-m$. Now we give an upper bound on the number of $r$-cliques in $G_{\textup{blue}}$. Let $d(v)$ be the degree of $v$ in $G_{\textup{blue}}$.
 Obviously the neighborhood of every vertex in $G_{\textup{blue}}$ is $F'$-free. An $F'$-free graph on $d(v)$ vertices contains at most $$\ex(d(v),K_{r-1},F')\le cd(v)$$ copies of $K_{r-1}$. Thus $v$ is contained in at most $c d(v)$ copies of $K_r$ in $G_{\textup{blue}}$. 
 If we sum, for each vertex, the number of copies of $K_r$ containing a vertex, then each $K_r$ is counted $r$ times.
On the other hand as $\sum_{v\in V(G_{\textup{blue}})} d(v)=2m$, we have $\sum_{v\in V(G_{\textup{blue}})} cd(v)=2cm$.
This gives that the number of $r$-cliques in $G_{\textup{blue}}$ is at most $2cm/r.$
Thus we obtain
$$
g_r(G) \le (\ex(n,F)-m)+ \frac{2c}{r}m 
 \le \max\left\{1, \frac{2c}{r} \right\} (\ex(n,F) - m +m) = \max\left\{1, \frac{2c}{r} \right\} \ex(n,F).
$$
Now, using Corollary~\ref{blue-redtobergeF}, the proof is complete.
\end{proof}

Note that Lemma \ref{genenew} gives an improvement (with a simpler proof) to a theorem of Gerbner, Methuku and Vizer (\cite{gmv}, Theorem 12) in the case when $c$ is a constant. We will only apply the lemma in this case. Our proof uses a special case of Claim 26 from \cite{gmv} about the number of $r$-cliques in $F'$-free graphs with $n$ vertices and $m$ edges. Another upper bound on this number is given in \cite{gmv}, which implies a better upper bound on $\ex_r(n,\textup{Berge-}F)$ in case $c$ is not a constant, i.e., in case when $F$ cannot be made acyclic by deleting a vertex.

\section{Berge trees}

Erd\H os-S\'os conjectured \cite{ErSo} that the extremal number for trees is the same as that of paths, i.e.,

\begin{conjecture}[Erd\H os-S\'os conjecture]
	Let $T$ be a tree on $k+1$ vertices. Then
    \[
    \ex(n,T) \leq \frac{k-1}{2}n.
    \]
\end{conjecture}

A proof of the conjecture for large trees was announced by Ajtai, Koml\'os, Simonovits and Szemer\'edi. The conjecture is known to hold for various classes of trees. In particular, the conjecture holds for paths by the Erd\H os-Gallai theorem \cite{Er-Ga} and more generally for spiders by a result of Fan, Hong and Liu \cite{faholi}. Recall that a \textit{spider} is a tree  with at most one vertex of degree greater than $2$.

Gy\H{o}ri, Katona and Lemons \cite{GyKaLe} generalized the Erd\H os-Gallai theorem to Berge-paths. More precisely, they determined $\ex_r(n,\textup{Berge-}P_k)$ for both the range $k>r+1$ and the range $k \le r$, where $P_k$ denotes a path of length $k$ (i.e., a path on $k+1$ vertices). 

\begin{thm}[Gy\H{o}ri, Katona,  Lemons \cite{GyKaLe}]
\label{GKL}
If $k>r+1>3$, then
\begin{displaymath}
\ex_r(n,\textup{Berge-}P_k) \le \frac{n}{k} \binom{k}{r}.
\end{displaymath}
%Moreover, when $k$ divides $n$, the result is sharp by partitioning the vertex set into sets of size $k$ and taking all possible subsets of size $r$ in each of these sets.
If $r \ge k>2$, then
\begin{displaymath}
\ex_r(n,\textup{Berge-}P_k) \le \frac{n(k-1)}{r+1}.
\end{displaymath}
%Moreover, when $r+1$ divides $n$, the result is sharp by partitioning the vertex set into sets of size $r+1$ and taking exactly $k-1$ subsets of size $r$ in each of these sets. 
\end{thm}
For the case $k=r+1$, Gy\H{o}ri, Katona and Lemons conjectured that the upper bound should have the same form as the $k>r+1$ case. This was settled by Davoodi, Gy\H{o}ri, Methuku and Tompkins \cite{DavoodiGMT} who showed that if
$k = r+1>2$, then
\begin{displaymath}
\ex_r(n,\textup{Berge-}P_k) \le \frac{n}{k} \binom{k}{r} = n.
\end{displaymath}

We generalize the above theorem to every tree and prove a sharp result under the assumption that Erd\H os-S\'os conjecture holds.

\begin{thm}\label{trees-thm}
Let us suppose that the  Erd\H os-S\'os conjecture holds for all trees.
Let $T$ be a tree on $k+1$ vertices. If $k > r+1>3$, then 
\[
\ex_r(n,\textup{Berge-}T) \leq \frac{n}{k}\binom{k}{r}.
\]
Moreover, if $k$ divides $n$, this bound is sharp.
If $k \leq r+1$, then
\[
\ex_r(n,\textup{Berge-}T) \leq \frac{k-1}{2}n.
\]
\end{thm}

\begin{proof}
	Let us fix $k-r$ and proceed by
	induction on $r$. Put $r=3$ and let us remove a leaf from $T$ to get a tree $T'$ on $k$ vertices.
	By the Erd\H os-S\'os conjecture we have $\ex(n,K_2,T') = \ex(n,T') \leq \frac{k-2}{2}n$.
    Thus, $c=\frac{k-2}{2}$ in the statement of Lemma~\ref{genenew}. First let us consider the case $k > r+1$. Then we have $\max\left\{\frac{2c}{r}, 1 \right\} = \max\left\{\frac{k-2}{r}, 1 \right\} = \frac{k-2}{r}$. So Lemma~\ref{genenew} gives
    \[
    \ex_3(n,\textup{Berge-}T) \leq \frac{k-2}{3}\ex(n,T) \leq \frac{k-2}{3} \cdot \frac{k-1}{2}n = \frac{n}{k} \binom{k}{3}
    \]
proving the base case.
Now assume that
\[
\ex_{r-1}(n,\textup{Berge-}T') \leq \frac{n}{k-1}\binom{k-1}{r-1}.
\]
By Proposition~\ref{Bergecontainment} this implies
\[
\ex(n,K_{r-1},T') \leq \frac{n}{k-1}\binom{k-1}{r-1}.
\]
Similar to the base case this gives $c=\frac{1}{k-1}\binom{k-1}{r-1}$ in the statement of Lemma~\ref{genenew}. Moreover,  we have $$\frac{2c}{r} = \frac{2}{r(k-1)}\binom{k-1}{r-1} = \frac{2}{k(k-1)}\binom{k}{r} = \frac{\binom{k}{r}}{\binom{k}{2}}.$$

Since $k > r+1$, we have $2 \le r \le k-2$, which means $\binom{k}{r} \ge \binom{k}{2}$, so $\max\left\{\frac{2c}{r}, 1 \right\} = \frac{2c}{r}$. So we may apply Lemma~\ref{genenew} again to get
    \[
    \ex_r(n,\textup{Berge-}T) \leq \frac{2}{r(k-1)}\binom{k-1}{r-1}\ex(n,T) \leq 
    \frac{2}{r(k-1)}\binom{k-1}{r-1} \frac{k-1}{2}n = \frac{n}{k} \binom{k}{r}.
    \]
    
When $k$ divides $n$, the result is sharp by considering a hypergraph obtained by partitioning the vertex set into sets of size $k$ and taking all possible subsets of size $r$ in each of these sets.

Let us continue with the case $k\le r+1$. The proof is similar to the previous case; we proceed by fixing $k-r$ and by applying induction on $r$. However, now $\max\left\{\frac{2c}{r}, 1 \right\} = 1$. Thus we obtain $\ex_3(n,\textup{Berge-}T) \leq \ex(n,T) \le \frac{k-1}{2}n$ proving the base case. Similarly in the induction step we get $c = \frac{k-2}{2}$, so $\max\left\{\frac{2c}{r}, 1 \right\} = 1$ again. Thus we obtain $\ex_r(n,\textup{Berge-}T) \leq \ex(n,T) \le \frac{k-1}{2}n$ (by Lemma~\ref{genenew} and the Erd\H os-S\'os conjecture), finishing the proof.
\end{proof}

Note that the above proof uses only the fact that the Erd\H os-S\'os conjecture holds for $T$ and its subtrees. In particular, this gives a new proof of Theorem \ref{GKL} in the case $k>r+1>3$ and a sharp result for spiders (as Erd\H os-S\'os conjecture holds for paths and spiders).

By Proposition~\ref{Bergecontainment} we have $\ex(n,K_r,T) \leq \ex(n,\textup{Berge-}T)$ so Theorem~\ref{trees-thm} gives the following corollary for the subgraph-counting problem.

\begin{cor}
Let us suppose that the  Erd\H os-S\'os conjecture holds for all trees.
Let $T$ be a tree on $k+1$ vertices. If $k>r+1$, then 
\[
\ex(n,K_r,T) = (1+o(1))\frac{n}{k} \binom{k}{r}.
\]
\end{cor}

When $k > r+1$, then Theorem~\ref{trees-thm} gives a sharp bound. When $k \le r$ we give the following result without assuming that the Erd\H os-S\'os conjecture is true. Our proof of the theorem below is inspired by some ideas in \cite{EGMNSTLongBerge}.

\begin{thm}\label{delt}
Let $T$ be a tree on $k+1$ vertices with maximum degree $\Delta(T)$. If $k \le r$, then we have $\ex_r(n,\textup{Berge-}T) \leq (\Delta(T)-1)n$.
\end{thm}
\begin{proof}
 We proceed by induction on $n$. The base case $n \le r+1$ is easy to check. Suppose the statement of the theorem is true for all values less than $n$ and let us show it is true for $n$. 

Let $\mathcal H$ be a $\textup{Berge-}T$-free hypergraph on $n$ vertices. We may assume that:
\begin{equation}
\label{eq:Multifoldhalls}
\text{Any set $S \subseteq V(\mathcal H)$ is incident to at least $(\Delta(T)-1)|S|$ hyperedges of $\mathcal H$.}
\end{equation}
 Indeed, otherwise we may delete the vertices of $S$ (and the hyperedges incident to them) from $\mathcal H$ to obtain a new hypergraph $\mathcal H'$ with $n-|S|$ vertices. Note that we deleted less than $(\Delta(T)-1)|S|$ hyperedges.  By the induction hypothesis, $\mathcal H'$ has at most $(\Delta(T)-1)(n-|S|)$ hyperedges. This implies that $\mathcal H$ has less than $(\Delta(T)-1)|S| + (\Delta(T)-1)(n-|S|) \le (\Delta(T)-1) n$ hyperedges, and we are done.

Now consider an auxiliary bipartite graph $A$ with parts $A_1$ and $A_2$, where
$A_1 = V(\mathcal H)$ and $A_2 = E(\mathcal H)$ and $v \in A_1$ is adjacent to $e \in A_2$ in $A$ if the hyperedge $e$ is incident to the vertex $v$ in $\mathcal H$. We will use the following claim. For a vertex $a \in A_1$, let $N(a) \subseteq A_2$ denote the set of vertices adjacent to $a$ in $A$, and for a set $S \subseteq A_1$, let $N(S) = \cup_{a \in S} N(a)$.

\begin{clm}
\label{private_hyperedges}
For every vertex $a \in A_1$, there exists a set of $\Delta(T)-1$ vertices $S_a \subseteq N(a)$, such that if $a \not = a'$ then $S_a \cap S_{a'} = \emptyset$. 
\end{clm}
\begin{proof}\renewcommand{\qedsymbol}{$\blacksquare$}
Note that \eqref{eq:Multifoldhalls} implies $|N(S)| \geq (\Delta(T)-1)|S|$ for any set $S \subseteq A_1$. Let us replace each vertex $a \in A_1$ with $\Delta(T)-1$ new vertices $a_1, a_2, \ldots, a_{\Delta(T)-1}$ so that each $a_i$ has the same neighborhood as $a$. Let $A'$ be the resulting bipartite graph with parts $A'_1$ and $A_2$ (note $|A'_1| = (\Delta(T)-1)|A_1|$). Then it is easy to see that for each $S \subseteq A'_1$, $|N(S)| \ge |S|$, so Hall's condition holds. Thus we can find a perfect matching $M$ in $A'$ that matches all of the vertices in $A'_1$.

Now consider an arbitrary vertex $a \in A_1$ of $A$. The corresponding vertices in $A'$ are $a_1, a_2, \ldots, a_{\Delta(T)-1}$, and let $a_i b_i \in M$ ($1 \le i \le\Delta(T)-1$)  be the edges of $A'$ that match these vertices. Let $S_a = \{b_1, b_2, \ldots, b_{\Delta(T)-1}\}$. Then it is easy to see that if $a \not = a'$ then $S_a \cap S_{a'} = \emptyset$, proving the claim.
\end{proof}

Claim \ref{private_hyperedges} shows that we can assign a set $S_v$ of $\Delta(T)-1$ ``private" hyperedges to each vertex $v \in V(\mathcal H)$ such that for any two distinct vertices $v, v'$,  the sets $S_v$ and $S_{v'}$ are disjoint. 

Since any tree contains a leaf vertex, we can order the vertices of $T$ as $v_1,v_2,\ldots,v_{k+1}$ such that $v_1$ is a leaf and every $v_i$ ($i\ge 2$) is adjacent to exactly one of the vertices $v_1,v_2, \ldots, v_{i-1}$, say $v_{i'}$. This vertex $v_{i'}$ will be called \emph{backward neighbor} of $v_i$.

Let $v_1$ be represented by an arbitrary vertex $u_1$ of $\mathcal H$, and consider a hyperedge $h \in S_{u_1}$ (note that $h$ is a hyperedge of $\mathcal H$ containing $u_1$). Let $v_2$ be represented by a vertex $u_2 \in h$ different from $u_1$.

Let the subtree of $T$ spanned by $v_1, \ldots, v_{i-1}$ be denoted by $T_{i-1}$. Now suppose we have already found a Berge copy of $T_{i-1}$ in $\mathcal H$ (for some $i \ge 3$), where the vertex $v_l$ of $T_{i-1}$ is represented by the vertex $u_l \in V(\mathcal H)$ in this Berge copy (for each $1 \le l \le i-1$). Moreover, suppose the Berge copy has the additional property that for each $j$ with $1 \le j \le i-1$, if $v_{j'}$ is the backward neighbor of $v_j$, then the edge $v_{j'}v_j$ is represented by a hyperedge in $S_{u_{j'}}$. 

We now wish to find a vertex $u_i$ in $\mathcal H$ to represent $v_i$ and obtain a Berge copy of the subtree $T_i$ spanned by $v_1, \ldots ,v_i$ such that if $v_{i'}$ is the backward neighbor of $v_i$ then the edge $v_{i'}v_i$ is represented by a hyperedge in $S_{u_{i'}}$. To this end, let $v_{i''}$ be the backward neighbor of $v_{i'}$. Note that $v_{i'}$ has at most $\Delta(T)-1$ neighbors among the vertices of $T_{i-1}$ (recall that $v_{i'}v_i$ is an edge of $T$). Since the hyperedges in $S_{u_{i'}}$ were only used to represent edges incident to $v_{i'}$ in $T$, and we assumed the edge $v_{i''}v_{i'}$ was represented by a hyperedge in $S_{u_{i''}}$, we obtain that at most $\Delta(T)-2$ hyperedges of $S_{u_{i'}}$ have been used to represent the edges of $T_{i-1}$. Hence (as $|S_{u_{i'}}| = \Delta(T)-1$) at least one hyperedge, say $h$, of $S_{u_{i'}}$ has not been used to represent any of the edges of $T_{i-1}$, so we can use it to represent the edge $v_{i'}v_i$ provided $h$ contains a vertex $u_i$ not in $T_{i-1}$ (which can be used to represent $v_i$) -- this is the case if $|h| > |V(T_{i-1})|$; this inequality holds whenever $r > k$ because $|h| = r$ and $|V(T_{i-1})| \le k$. 

It only remains to deal with the case $r = k$; in this case the inequality $|h| > |V(T_{i-1})|$ does not hold only when $|V(T_{i-1})| \ge k$, so when $i = k+1$ (i.e., when we want to embed the last vertex $v_{k+1}$ of the tree). Let $v_{i}$ be the backward neighbor of $v_{k+1}$. Then, as we argued before, we can find a hyperedge $h$ in $S_{v_i}$ that has not been used to represent any edges of $T_k$. This hyperedge $h$ cannot contain any vertex $u_{k+1} \not \in V(T_k)$, because otherwise we can use $u_{k+1}$ to represent $v_{k+1}$ and we have found a Berge-$T$ in $\mathcal H$, a contradiction. So $h = V(T_k)$.  Now consider the backward neighbor $v_{i'}$ of $v_i$; then we know the edge $v_{i'}v_i$ was represented by a hyperedge $h' \in S_{u_{i'}}$. Note that $h \not = h'$ since $S_{v_{i'}} \cap S_{v_{i}} = \emptyset$. So $h'$ must contain a vertex $u \not \in V(T_k)$. Moreover both $h$ and $h'$ contain the edge $v_{i'}v_i$, so we redefine the edge $v_{i'}v_i$ to be represented by $h$ and we use $h'$ to represent the edge $v_iv_{k+1}$, where the vertex $u \in V(\mathcal H)$ represents the vertex $v_{k+1}$. This gives us the desired Berge copy of $T$ in $\mathcal H$; a contradiction. 
\end{proof}

Now we prove and upper bound on $\ex_r(n,\textup{Berge-}T)$ for every uniformity $r$ and every tree $T$ without the need for Erd\H os-S\'os conjecture.

\begin{proposition}
\label{without_using_ErdosSos}
Let $T$ be a tree on $k+1$ vertices. If $k>r$, then
\[
\ex_r(n,\textup{Berge-}T) \leq \frac{2(r-1)}{k}\binom{k}{r}n.
\]

If $k\le r$, then 
\[
\ex_r(n,\textup{Berge-}T)\le (k-1)n.
\]
\end{proposition}

\begin{proof}  Let us remove a leaf from $T$ to get a tree $T'$ on $k$ vertices. 

We will show that $\ex(n,K_{r-1},T')\le \binom{k-2}{r-2}n$ by induction on $n$. Assume the statement is true for $n-1$ and prove it for $n$. Let $G$ be a $T'$-free graph on $n$ vertices. First we claim that there is a vertex $v$ of degree at most $k-2$. Indeed, otherwise we can embed $T'$ greedily into $G$. Thus the number of copies of $K_{r-1}$ containing $v$ is at most $\binom{k-2}{r-2}$. We delete $v$ (and the edges containing it) to obtain a graph $G'$ on $n-1$ vertices. By induction, $G'$ contains at most $\binom{k-2}{r-2} (n-1)$ copies of $K_{r-1}$. Thus the number of copies of $K_{r-1}$ in $G$ is at most $\binom{k-2}{r-2} (n-1) + \binom{k-2}{r-2} = \binom{k-2}{r-2} n$, as desired. 

So we can choose $c = \binom{k-2}{r-2}$ in Lemma \ref{genenew} to obtain that 
\[\ex_r(n,\textup{Berge-}T) \leq \max\left\{\frac{2}{r} \binom{k-2}{r-2}, 1 \right\} \ex(n,T).\]

Suppose $k > r$. Then $\binom{k-2}{r-2}\ge r-1$, which implies $\frac{2}{r} \binom{k-2}{r-2} \ge \frac{2}{r} (r-1) \ge 1$, where the last inequality holds because $r \ge 2$. Therefore, $\max\left\{\frac{2}{r} \binom{k-2}{r-2}, 1 \right\}  \ex(n,T) = \frac{2}{r} \binom{k-2}{r-2} \ex(n,T)$. It is well known (and easy to see) that $\ex(n,T)  \le (k-1)n$. Thus,
$$\ex_r(n,\textup{Berge-}T) \leq \frac{2}{r} \binom{k-2}{r-2} (k-1)n = \frac{2(r-1)}{k}\binom{k}{r}n.$$

On the other hand, if $k \le r$, then $\max\left\{\frac{2}{r} \binom{k-2}{r-2}, 1 \right\} = 1$. So by Lemma \ref{genenew}, we have $\ex_r(n,\textup{Berge-}T) \le \ex(n,T) \le (k-1)n$. 
\end{proof}

Let us finish this section by considering stars. Let $S_k$ denote the star with $k$ edges. 

\begin{thm}
If $k > r+1$, then
\[
\ex_r(n,\textup{Berge-}S_k)\le \frac{n}{k}\binom{k}{r}.
\]
Moreover, this bound is sharp whenever $k$ divides $n$. 

If $k\le r+1$, then
\[
\ex_r(n,\textup{Berge-}S_k)\le \left \lfloor \frac{n(k-1)}{r} \right \rfloor.
\]
Moreover, this bound is sharp whenever $n$ is large enough. 
\end{thm}

\begin{proof}

First let us consider the case $k > r+1$. As it is known that Erd\H os-S\'os conjecture holds for stars, Theorem \ref{trees-thm} gives the desired (sharp) bound. 

Now consider the case when $k\le r+1$. Every vertex in a Berge-$S_k$-free graph has degree at most $k-1$. Indeed, assume $v$ is contained in the hyperedges $e_1,\dots, e_k$. Let us consider the auxiliary bipartite graph where part $A$ consists of the $(r-1)$-sets $e_1\setminus \{v\},\dots, e_k\setminus \{v\}$, and part $B$ consists of the vertices contained in these sets. We connect a vertex in $A$ to the vertices in $B$ that are contained in the corresponding $(r-1)$-set. It is easy to see that a matching covering $A$ would give us a Berge-$S_k$. If there is no such matching, then by Hall's condition there is a subset $A'$ of $A$ with $|A'|>|N(A')|$, where $N(A')$ denotes the set of neighbors of $A'$ in $B$. As every vertex in $A$ is connected to $r-1$ vertices, we have $|A'|>|N(A')|\ge r-1 \ge 1$. As two different $(r-1)$-sets together contain at least $r$ vertices, we obtain $|A'|>|N(A')|\ge r$, thus $|A'| \ge r+1 \ge k$, so $|A'| = r+1 = k$. However, $r+1$ different $(r-1)$-sets together contain at least $r+1$ vertices, a contradiction.
This gives the upper bound $\ex_r(n,\textup{Berge-}S_k)\le \lfloor n(k-1)/r\rfloor$.
It is easy to see that for large enough $n$, there exist $r$-uniform hypergraphs on $n$ vertices such that less than $r$ vertices have degree $k-2$ and the remaining vertices have degree $k-1$. This gives the desired lower bound.
\end{proof}

\section{Berge-\texorpdfstring{$K_{2,t}$}{K2t}}

Gerbner, Methuku and Vizer \cite{gmv} showed that if $t \geq 7$, then
\[
\ex_3(n,\textup{Berge-}K_{2,t})=(1+o(1))\frac{1}{6}(t-1)^{3/2}n^{3/2}.
\]

They also gave bounds for higher uniformities. Using Lemma~\ref{genenew}, we show that the same result holds for $t = 4,5,6$ as well, and also improve their bounds for higher uniformities as follows.

\begin{thm} If $t \geq r+1$, then $$\ex_r(n,\textup{Berge-}K_{2,t})\le(1+o(1)) \frac{\sqrt{(t-1)}\binom{t}{r-1}}{r  t} n^{3/2}.$$

If $t\le r$, then $$\ex_r(n,\textup{Berge-}K_{2,t})\le (1+o(1)) \frac{\sqrt{t-1}}{2} n^{3/2}.$$

In particular, if $r=3$ and $t\ge 4$, then 
\[
\ex_3(n,\textup{Berge-}K_{2,t})=(1+o(1))\frac{1}{6}(t-1)^{3/2}n^{3/2}.
\]
\end{thm}

\begin{proof}
We apply Lemma~\ref{genenew} with $F = K_{2,t}$, and $F' = K_{1,t}$.
In a $K_{1,t}$-free graph, since the degree of any vertex is at most $t-1$, there are at most $\binom{t-1}{r-2}$ cliques of size $r-1$ containing any vertex. Therefore, we get the following.
$$\ex(n, K_{r-1}, K_{1,t}) \le \frac{n}{r-1}\binom{t-1}{r-2}= \frac{n}{t} \binom{t}{r-1}.$$

\noindent
Thus $c = \frac{1}{t}\binom{t}{r-1}$ in Lemma~\ref{genenew}. We have $\max\{2c/r,1\}=2c/r$ if and only if $t\ge r+1$. Thus Lemma~\ref{genenew} gives 
\[
\ex_r(n,\textup{Berge-}K_{2,t})\le \frac{2c}{r}\ex(n,K_{2,t}) = \frac{2}{r t} \binom{t}{r-1}\ex(n,K_{2,t})
\]
when $t \geq r+1$, and $\ex_r(n,\textup{Berge-}K_{2,t})\le \ex(n,K_{2,t})$ if $t\le r$. 

Now using a result of F\"uredi \cite{F1996} which states $\ex(n,K_{2,t}) \le  (1+o(1)) \frac{\sqrt{t-1}}{2} n^{3/2} $, the proof is complete.
\end{proof}

\section{Berge-\texorpdfstring{$C_{2k}$}{C2k}}

   Gerbner, Methuku and Vizer \cite{gmv} improved earlier bounds due to F\"uredi and \"Ozkahya \cite{FO2017} by showing that if $k \ge 5$, then 
\[\ex_3(n,\textup{Berge-}C_{2k})\le \frac{2k-3}{3} \ex(n,C_{2k}).\]

Now using Lemma~\ref{genenew}, we show that the same statement holds for $k\ge 3$, that is, it holds for $C_6$ and $C_8$ as well.

\begin{thm}
\label{2kcycle}
If $k\ge 3$, then 
\[
\ex_3(n,\textup{Berge-}C_{2k})\le \frac{2k-3}{3} \ex(n,C_{2k}).
\]
\end{thm}
\begin{proof}
We apply Lemma~\ref{genenew} with $F = C_{2k}$, and $F' = P_{2k-2}$ (a path of length $2k-2$). The Erd\H{o}s-Gallai theorem implies that $\ex(n, K_2, P_{2k-2}) \le \frac{2k-3}{2}n$. Then $c = \frac{2k-3}{2}$ in the statement of Lemma~\ref{genenew}. Moreover, $2c/3 \ge 1$ whenever $k \ge 3$. Thus Lemma~\ref{genenew} gives that $$\ex_3(n,\textup{Berge-}C_{2k})\le \frac{2c}{3} \ex(n,C_{2k}) =  \frac{(2k-3)}{3} \ex(n,C_{2k})$$ 
whenever $k \geq 3.$
\end{proof}

Note that for larger $r$, Jiang and Ma \cite{JM2016} proved $\ex_r(n,\textup{Berge-}C_{2k}) \le O_r(k^{r-2})  \ex(n,C_{2k}).$ In \cite{gmv} Gerbner, Methuku and Vizer gave a different proof of this result with an improved constant factor. With a similar calculation to that in the proof of Theorem \ref{2kcycle},  we can again reprove this result with an improved constant factor using Lemma~\ref{genenew} but for a larger range of $k$.

\section{Berge theta graphs}

A theta graph $\Theta_{k,t}$ is the graph of $t$ internally-disjoint paths of length $k$ between
a fixed pair of vertices. When $t=2$ the theta graph $\Theta_{k,2}$ is exactly the even cycle $C_{2k}$. An upper-bound of $C_{k,t} n^{1+1/k}$ (for some constant $C_{k,t}$ depending only on $k$ and $t$) on the extremal number of $\Theta_{k,t}$ is given by Faudree and Simonovits \cite{FS}. A lower bound is given by Conlon \cite{Conlon}. Recently, He and Tait \cite{HT} generalized the upper bound to the Berge setting.

\begin{thm}[He, Tait \cite{HT}]\label{hetait}

	\[
    \ex_r(n,\textup{Berge-}\Theta_{k,t})  \le M_{k,t,r,2} \cdot \ex(n,\Theta_{k,t}) =  O(n^{1+1/k}).
    \]
    where $M_{k,i,r,m} = \sum_{j=1}^{k+1} \binom{mk(i-1)+jm-m}{r-m} + k + 1$.
\end{thm}
They also showed that for fixed $r$ and any $k \ge  2$, there exists $t$ such that the above upper bound is sharp in the order of magnitude.

\vspace{2mm}

Now we improve the constant factor in Theorem \ref{hetait}.

\begin{thm}\label{improvedtheta}
\[ \ex_r(\textup{Berge-}\Theta_{k,t}) \leq \begin{cases} 
     \frac{2}{r(r-1)} \binom{(k-1)t-1}{r-2} \ex(n,\Theta_{k,t}) & \textup{if } (k-1)t > r \\[1em]
      \frac{(k-1)t-1}{(k-1)t}\ex(n,\Theta_{k,t}) & \textup{if } (k-1)t = r \\[1em]
       \frac{2(t-1)}{r}\ex(n,\Theta_{k,t}) & \textup{if } (k-1)t < r
   \end{cases}
\]

\end{thm}

\begin{proof}
First suppose $(k-1)t > r$. Observe that we can remove a vertex from a theta graph $\Theta_{k,t}$ to get a spider $T$ on $(k-1)t+1$ vertices. Since the Erd\H{o}s-S\'os conjecture is known to hold for spiders, Theorem~\ref{trees-thm} combined with Lemma~\ref{genenew} implies that $\ex(n,K_{r-1},T)  \le \frac{1}{(k-1)t} \binom{(k-1)t}{r-1}n$. Therefore, by Lemma~\ref{genenew} we have 
\begin{align*}
\ex_r(n,\textup{Berge-}\Theta_{k,t}) & \le \frac{2}{r(k-1)t} \binom{(k-1)t}{r-1}\ex(n,\Theta_{k,t}) 
 = \frac{2}{r(r-1)} \binom{(k-1)t-1}{r-2}\ex(n,\Theta_{k,t}).
\end{align*}

Suppose now $(k-1)t = r$. Similarly combining Theorem~\ref{trees-thm} with Lemma~\ref{genenew}, we obtain $\ex_r(n,\textup{Berge-}\Theta_{k,t})\le \frac{(k-1)t-1}{(k-1)t}\ex(n,\Theta_{k,t})$.
Finally suppose $(k-1)t < r$. Then combining Theorem \ref{delt} with Lemma~\ref{genenew}, we obtain $\ex_r(n,\textup{Berge-}\Theta_{k,t})\le \frac{2(t-1)}{r}\ex(n,\Theta_{k,t})$.
\end{proof}

Note that an upper bound of $O(n^{1+1/k})$ in Theorem \ref{hetait} also follows from a result of Gerbner, Methuku and Vizer \cite{gmv} which states that if $F$ contains a vertex such that deleting it makes $F$ acyclic, then $\ex_r(n,\textup{Berge-}F)=O(\ex(n,F))$. Now we will reprove this result from \cite{gmv} using Lemma~\ref{genenew} to give an improved constant factor.

\begin{thm} Let $F$ be a graph on $k$ vertices and $v$ be one of its vertices such that deleting $v$ from $F$ we obtain a forest $F'$. 
\[ \ex_r(n,\textup{Berge-}F) \leq \begin{cases} 
     \frac{4(r-2)}{(r-1)r}\binom{k-3}{r-2}\ex(n,F) & \textup{if } k>r+1 \\[1em]
      \frac{2(k-3)}{r}\ex(n,F) & \textup{if } \frac{r}{2}+3<k\le r+1 \\[1em]
       \ex(n,F) & \textup{if } k\le \frac{r}{2}+3
   \end{cases}
\]
\end{thm}

\begin{proof} If $k>r+1$, then Proposition \ref{without_using_ErdosSos} implies that $\ex(n,K_{r-1},F')\le\ex_{r-1}(n,\textup{Berge-}F')\le \frac{2(r-2)}{k-2}\binom{k-2}{r-1}n$. Thus if $k>r+1$, we use Lemma~\ref{genenew} with $c = \frac{2(r-2)}{k-2}\binom{k-2}{r-1}$. If $k>r+1$, then it is easy to see that $\max\{\frac{2c}{r},1\}=\frac{2c}{r} = \frac{4(r-2)}{(k-2)r}\binom{k-2}{r-1} = \frac{4(r-2)}{(r-1)r}\binom{k-3}{r-2}$, so Lemma~\ref{genenew} gives the desired bound.

If $k\le r+1$, then Proposition \ref{without_using_ErdosSos} implies that $\ex(n,K_{r-1},F')\le\ex_{r-1}(n,\textup{Berge-}F')\le (k-3)n$, thus we can use Lemma~\ref{genenew} with $c = k-3$. If $\frac{r}{2}+3<k\le r+1$, then $\max\{\frac{2c}{r},1\}=\frac{2c}{r}$, while if $k\le \frac{r}{2}+3$, then $\max\{\frac{2c}{r},1\}=1$. In both cases Lemma~\ref{genenew} gives the desired bound.
\end{proof}

\section{Berge-\texorpdfstring{$K_r$}{Kr}}

In this section, for brevity, we use the term {\it $r$-graph} to refer to an $r$-uniform hypergraph.
Let $T_r(n,k)$ be the complete $k$-partite $n$-vertex $r$-graph where all the parts have size $\lfloor n/k\rfloor$ or $\lceil n/k \rceil$. Erd\H os \cite{erd} showed that the Tur\'an graph $T_2(n,k-1)$ maximizes not only the number of edges among $n$-vertex $K_k$-free graphs, but also the number of $K_r$'s for any $r<k$.

\begin{thm}[Erd\H os, \cite{erd}]\label{erd} For any $k, r$ and $n$,
\[
\ex(n,K_r,K_k)=\mathcal{N}(T_2(n,k-1),K_r).
\]

\end{thm}

Observe that if $2<r<k$, then $\ex(n,K_r,K_k)$ is at least cubic, so Proposition \ref{Bergecontainment} gives asymptotically tight bounds on $\ex_r(n,\textup{Berge-}K_k)$. In this paper we are interested in exact results for every $n$. For $r < k$, let us define the threshold $n_0=n_0(k,r)$ to be the smallest possible integer such that $T_r(n,k-1)$ is the largest Berge-$K_k$-free $r$-graph for every $n\ge n_0$. We will see that $n_0$ exists.

\vspace{2mm}

The expansion $F^{+r}$ of a graph $F$ is an $r$-uniform hypergraph obtained by adding $r-2$ distinct new vertices to each edge of $F$. Pikhurko \cite{pik}, improving an asymptotic result of Mubayi \cite{mub}, showed that for $r<k$ there is $n_1=n_1(k,r)$ such that the largest $K_k^{+r}$-free hypergraph is $T_r(n,k-1)$ for $n\ge n_1$. Observe that the expansion of $F$ is one specific Berge copy of $F$. Thus Pikhurko's result shows $T_r(n,k-1)$ is the largest Berge-$K_k$-free $r$-graph if $n\ge n_1$ provided $r<k$, so $n_0$ exists and is at most $n_1$. However, the value of $n_1$ which follows from Pikhurko's proof is quite large.

\vspace{2mm}

Our goal is to give better bounds on the threshold $n_0(k,r)$.

In this direction, research has been carried out for $3$-graphs: Maherani and Shahsiah \cite{masha} showed that for $k\ge 13$, we have $n_0(k,3)=0$, i.e., for every $n$, $T_3(n,k-1)$ contains the largest number of hyperedges among all Berge-$K_k$-free $3$-graphs, provided $k\ge 13$. Gy\'arf\'as \cite{gyarfas} proved that $n_0(4,3)=6$. The situation is different if $n < n_0(4,3) = 6$. In this case Gy\'arf\'as showed $\ex_3(5,\textup{Berge-}K_4)=5$; moreover, any $3$-uniform hypergraph on $5$ vertices with $5$ hyperedges shows that this bound is sharp (because we need at least $6$ hyperedges to form a $\textup{Berge-}K_4$). For $n \le 4$, trivially $\ex_3(n, \textup{Berge-}K_4) = \binom{n}{3}$ because a complete $3$-graph on at most $4$ vertices is $\textup{Berge-}K_4$-free. 

\vspace{2mm}

Our theorem below implies most of the results for $3$-graphs mentioned above and provides new bounds for all uniformities $r$.

\begin{thm}\label{compl} For any $n$, $k$ and $r$, we have $$\ex_r(n, \textup{Berge-}K_k) \le \max\{\ex(n, K_k), \ex(n, K_r, K_k)\}.$$
\end{thm}

Our proof is based on a careful adaptation of Zykov's symmetrization method (see \cite{zyk}) to red-blue graphs.

\begin{proof}[Proof of Theorem \ref{compl}] 
We will show that if $G$ is an $n$-vertex $K_k$-free graph $G$, then $g_r(G)$ is maximized when $G$ is the Tur\'an graph $T_2(n,k-1)$ with all edges of the same color.
This implies $g_r(G)$ is at most $\max\{\ex(n, K_k), \ex(n, K_r, K_k)\},$ so applying Lemma \ref{main} completes the proof.

Let $\cG$ be the family of $n$-vertex $K_k$-free red-blue graphs that maximize $g_r$.
Let $v_1,v_2,\dots, v_n$ be the vertex set of each of these graphs.
Then let $\cG'$ be the subfamily of red-blue graphs in $\cG$ with the maximum number of edges. 
Let $\cG''$ be the subfamily of graphs in $\cG'$ with the maximum number of red edges. Let $d_{\textup{red}}(v)$ denote the number of red edges incident to a vertex $v$.
Let $\cG_1$ be the subfamily of red-blue graphs in $\cG''$ which maximize $d_{\textup{red}}(v_1)$. We recursively define further subfamilies. For $2\le i\le n$, let $\cG_i$ be the subfamily of red-blue graphs in $\cG_{i-1}$ which maximize $d_{\textup{red}}(v_{i})$.

\begin{clm}\label{symm} 
Any red-blue graph $G\in \cG_n$ is a complete multipartite graph such that for any pair of classes $A,B$ all edges between $A$ and $B$ are of the same color.
\end{clm}

\begin{proof}[Proof of Claim \ref{symm}]\renewcommand{\qedsymbol}{$\blacksquare$}
Let $G$ be an arbitrary red-blue graph in $\cG_n$.
For two non-adjacent vertices $u$ and $v$ of $G$ let us create a new graph $G'$ by 
deleting all edges incident to $u$ and adding new edges that join $u$ to the neighbors of $v$. Moreover, we color each new edge $uw$ with the same color as the edge $vw$.
We call this procedure {\it symmetrization} and say that we \textit{symmetrize $u$ to $v$}. 

We claim that $G'$ is $K_k$-free. Indeed, a copy of $K_k$ in $G'$ must contain a new edge and therefore must contain the vertex $u$. As $u$ and $v$ are non-adjacent, the $K_k$ does not include $v$. However, since $u$ and $v$ have the same neighborhood in $G'$ this means that there is a copy of $K_k$ containing $v$ in $G$; a contradiction. This implies that $G$ remains $K_k$-free under symmetrization.

For a vertex $v$, let $d^*(v)$ denote the number of red edges incident to $v$ plus the number of blue $r$-cliques of $G$ containing $v$. When we symmetrize $u$ to $v$, the number of red edges plus the number of blue $r$-cliques decreases by $d^*(u)$ and then increases by $d^*(v)$. Since $G\in \cG$, we have that the value of $g_r(G)$ is maximal. Thus, $d^*(u)\ge d^*(v)$. 
As we could also symmetrize $v$ to $u$, we must have $d^*(u)=d^*(v)$.

Similarly, as $G\in \cG'$ we must have that $d(u)=d(v)$ as otherwise we can symmetrize $u$ to $v$ (or $v$ to $u$) to get a graph with more edges. This would imply that $G \not \in \cG'$; a contradiction. A similar argument combined with the fact that $G \in \cG''$ implies that $d_{\textup{red}}(u)=d_{\textup{red}}(v)$ .

Now we show that $G$ is a complete multipartite graph. Assume not, then it is easy to see that we have three vertices $x,y,z$ such that $y$ and $z$ are adjacent, but $x$ is adjacent to neither $y$ nor $z$. By the previous paragraph, we have $d^*(y)=d^*(x)=d^*(z)$ and $d(y)=d(x)=d(z)$. Now we symmetrize $y$ to $x$ to obtain $G'$ and then $z$ to $x$ to obtain $G''$. Note that in the first symmetrization step $g_r$ and $d^*(y)$ do not change, while $d^*(z)$ does not increase (it might decrease if the edge $yz$ is red or contained in a blue $K_r$). This implies $g_r(G)=g_r(G')\le g_r(G'')$, which gives $g_r(G)=g_r(G')= g_r(G'')$ as $G\in \cG$. Similarly, in the first symmetrization step the total number of edges and $d(y)$ does not change, but this time $d(z)$ decreases by one. Thus in the second symmetrization step the total number of edges increases, a contradiction to the assumption that $G$ is in $\mathcal G'$.

Thus we obtained that $G$ is a complete multipartite graph, so any two non-adjacent vertices $u$ and $v$ have the same neighborhood. Let us define 
\[X(u,v)=\{x\in V: ux \text{ and $vx$ are of different colors}\}.\] 
Assume  $X(u,v)$ is non-empty and let $i$ be the smallest index with $v_i\in X(u,v)$. Without loss of generality $v_i$ is connected to $u$ by a red edge. Then we symmetrize $v$ to $u$. By the above observations, $g_r$, the total number of edges and the total number of red edges does not change. Also $d_{\textup{red}}(u)$ and $d_{\textup{red}}(v)$ do not change, and $d_{\textup{red}}(x)$ does not change for every $x\not\in X(u,v)$. In particular, $d_{\textup{red}}(v_j)$ does not change for $j<i$, showing $G'\in \cG_{i-1}$. But $d_{\textup{red}}(v_i)$ increases, contradicting our choice of $G$.

This finishes the proof of the claim. Indeed, assume $xy$ is red and $x'y'$ is blue such that $x$ and $x'$ are in the same class and $y$ and $y'$ are together in a different class. If the edge $xy'$ is blue, then $X(x,x')$ is non-empty, while if $xy'$ is red, then $X(y,y')$ is non-empty.
\end{proof}

Observe that if $u$ and $v$ are in the same part $A$ of a graph $G\in \cG_n$, then $d^*(u)=d^*(v)$. Let this value be denoted by $d^*(A)$. Similarly we have $d_{\textup{red}}(u) = d_{\textup{red}}(v)$, so denote this value by $d_{\textup{red}}(A)$. 

Note that a red-blue graph has an \textit{underlying} uncolored graph with the same vertex and edge set. 
Let $G_0$ be an arbitrary red-blue graph in $\cG_n$. Note that $G_0$ is a complete multipartite graph with  classes $A_1,\dots,A_j$ (note that $j \le k-1$ as $G_0$ is $K_k$-free). There may be several red-blue graphs in $\cG''$ with the same underlying graph $G_0$; let $\cH$ denote the family of such red-blue graphs.
By Claim~\ref{symm} for any red-blue graph in $\cH$ all edges between a pair of classes have the same color.
 Let $\cH_1$ denote the subfamily of those graphs in $\cH$ which maximize $d_{\textup{red}}(A_1)$. We recursively define further subfamilies. For $2\le i\le j$, let $\cH_i$ be the subfamily of graphs in $\cH_{i-1}$ which maximize $d_{\textup{red}}(A_{i})$. 

\begin{clm}\label{equi}
In any graph $G\in \cH_{j}$, being connected by red edges is an equivalence relation.
\end{clm}

\begin{proof}[Proof of Claim~\ref{equi}]\renewcommand{\qedsymbol}{$\blacksquare$}
Let us consider two parts $A$ and $B$ of $G$ that are connected by by red edges. We define another \textit{symmetrization step} as follows: For every part $C$ (distinct from $A$ and $B$), we change the color of the edges between $A$ and $C$ to the color of the edges between $B$ and $C$; we refer to this symmetrization step by saying that we \textit{symmetrize $A$ to $B$}. 
Note that the underlying graph $G_0$ does not change. In this way, the number of red edges plus blue $r$-cliques decreases by $|A|d^*(A)-|A||B|$  and then increases by $|A|d^*(B)-|A||B|$. This implies that $d^*(A)\ge d^*(B)$ as $G\in \cG$. As we can symmetrize $B$ to $A$ we obtain that $d^*(A) = d^*(B)$. Similarly, $d_{\textup{red}}(A)=d_{\textup{red}}(B)$ because if $d_{\textup{red}}(A)<d_{\textup{red}}(B)$ then symmetrizing $A$ to $B$ would increase the number of red edges while $g_r$ and the number of edges does not change, contradicting our assumption that $G\in \cG''$.

Now we show that being connected by red edges is an equivalence relation. Assume for a contradiction that there are two parts $A$ and $B$ connected by red edges and we have at least one other part connected by red edges to one of them and by blue edges to the other, and let $A_i$ be such a part with the smallest index $i$. Without loss of generality $A_i$ is connected to $A$ by red edges and to $B$ by blue edges, then we symmetrize $B$ to $A$. The resulting graph $G'$ is in $\cH$ as the underlying graph $G_0$ and $g_r$ do not change, and for any two classes of $G'$ the edges between them are of the same color. Also $d_{\textup{red}}(A)$ and $d_{\textup{red}}(B)$ do not change and $d_{\textup{red}}(A_{j})$ does not change for $j<i$. This shows $G'\in \cH_{i-1}$, but $d_{\textup{red}}(A_i)$ increases, showing $G$ cannot be in $\cH_i$; a contradiction. 
\end{proof}

Thus we found a red-blue graph $G$ that is complete multipartite, for any two of its classes the edges between them are of the same color, being connected by red edges in $G$ is an equivalence relation, and  $G$ maximizes $g_r$ among $K_k$-free red-blue graphs. We will show that all the edges of $G$ are of the same color.

If there are no red edges in $G$, we are done. Let $A$ and $B$ be classes connected by red edges. Let us assume first that the vertices of $A$ are not in a blue $r$-clique. Then we can change all the edges incident to $A$ to red. If there was any change, $g_r$ increases, which would be a contradiction. Thus all the edges incident to $A$ are red in $G$, but then all the edges in $G$ are red by Claim \ref{equi} and we are done. 

Hence there is a blue $K_r$ intersecting $A$. Thus there are at least $r-1$ classes $B_1,\dots,B_{r-1}$ in $G$ that are connected to $A$ and each other by blue edges. Note that $B\neq B_i$ for any $i$ as B is connected to $A$ by red edges. Then for any $i$, $B_i$ is connected to $B$ by blue edges, by applying Claim \ref{equi}. Let us now change all the edges between $A$ and $B$ to blue, and let $G'$ be the resulting graph. We claim that this way we delete $|A||B|$ red edges and add at least $(r-1)|A||B|$ blue $r$-cliques, thus $g_r$ increases, a contradiction. To prove this claim, let us pick one vertex from $r-2$ parts among the $B_i$'s, one vertex from $A$ and one vertex from $B$. This way we obtain a new blue $r$-clique. There are at least $r-1$ ways to pick $r-2$ $B_i$'s and one vertex from each of them. There are $|A||B|$ ways to pick the remaining two vertices from $A$ and $B$.

We obtained that $G$ is monochromatic, thus we have $g_r(G)\le \max\{\ex(n,K_k),\ex(n,K_r,K_k)\}$. Since $G$ maximizes $g_r$, the proof is complete.
\end{proof}

The following corollary of Theorem \ref{compl} determines $\ex_r(n, \textup{Berge-}K_k)$ exactly for every $n$ for any $k>r+2$.

\begin{corollary}
\label{nothing}
Let $r \ge 2$. If $k>r+2$, then $T_r(n,k-1)$ has the maximum number of hyperedges among all Berge-$K_k$-free $r$-graphs for every $n$.
\end{corollary}

\begin{proof}  If $n<k$, then the statement is trivial. For $n\ge k$, we are going to show that the Tur\'an graph $T_2(n,k-1)$ contains more copies of $K_r$ than edges. This statement together with Theorem \ref{compl} implies our corollary. To prove this statement we use induction on $n$. Consider the base case $n=k$. In this case observe that the Tur\'an graph $T_2(k,k-1)$ contains exactly two vertices in one part and exactly one vertex in each of the other parts. The number of edges is $\binom{k-2}{2}+2(k-2)$ and the number of copies of $K_r$ is $\binom{k-2}{r}+2\binom{k-2}{r-1}$. It is easy to see the latter is at least the former if $k > r+2$. 

Let us assume the statement holds for $n$, and the Tur\'an graph $T_2(n,k-1)$ has parts $A_1,\dots,A_{k-1}$. We add one more vertex $v$ to, say, part $A_{k-1}$, to obtain the Tur\'an graph $T_2(n+1,k-1)$. Let $G$ denote the subgraph induced by the other parts $A_1,\dots,A_{k-2}$. The number of edges added (by adding $v$) is the number of vertices of $G$, say $n'$, while the number of $r$-cliques added is the number of $(r-1)$-cliques in $G$. Observe that $G$ is the Tur\'an graph $T_2(n',k-2)$. Thus it is enough to prove that the number of $(r-1)$-cliques in the Tur\'an graph $T_2(n',k-2)$ is at least $n'$.  Once again, we can prove this statement by induction on the number of vertices. Note that $n'\ge k-2>r$. For the base case $n' = k-2$ the Tur\'an graph $T_2(n',k-2)$ is a complete graph, so it has $\binom{k-2}{r-1} $ copies of $K_{r-1}$, and it is easy to see that $\binom{k-2}{r-1} \ge k-2$ as $k>r+2$.
For the induction step, if we add any vertex, the number of $(r-1)$-cliques increases by at least one, finishing the proof.
\end{proof}

Note that if $k\le r$, then the Tur\'an hypergraph is empty. If $k=r+1$ or $k=r+2$ and $r\ge 3$, then $\ex_r(n, \textup{Berge-}K_k)$ is not given by the Tur\'an hypergraph $T_r(n,k-1)$ for small $n$, for example when $n=k$. Indeed, if $n=k=r+1$, then the Tur\'an hypergraph $T_r(n,k-1)$ contains two hyperedges, while even the complete $r$-uniform hypergraph on $n=r+1$ vertices does not contain a Berge-$K_k$ and it has $r+1>2$ hyperedges. If $n=k=r+2$, then the Tur\'an hypergraph $T_r(n,k-1)$ contains $2r+1$ hyperedges, while any $r$-uniform hypergraph on $n$ vertices with $\binom{k}{2}-1>2r+1$ hyperedges does not contain a Berge-$K_k$.

\vspace{2mm}
However, it is not hard to compute the upper bound that Theorem \ref{compl} gives on the thresholds $n_0(r+1,r)$ and $n_0(r+2,r)$ for any fixed $r$. One can easily see that the upper bound we obtain this way on $n_0(r+1,r)$ is $r+2\log r + O(1)$, while the upper bound on $n_0(2+1,r)$ is $r+\log r + O(1)$ as $r$ increases.

\vspace{2mm}

Let us finish this section by considering the $3$-uniform case. Theorem \ref{compl} and Corollary \ref{nothing} imply the following bounds for $3$-graphs: $n_0(4,3)\le 9$, $n_0(5,3)\le 7$, and for $k\ge 6$ we have $n_0(k,3)=0$.

For $\textup{Berge-}K_5$, by a simple (but tedious) case-analysis, one can show that $n_0(5,3) = 6$. The situation is different when $n < n_0(5,3) = 6$: Firstly,  $\ex_3(5, \textup{Berge-}K_5) = 9$ since any $3$-uniform hypergraph on $5$ vertices with $9$ hyperedges is Berge-$K_5$-free. Now if $n < 5$, $\ex_3(n, \textup{Berge-}K_5) = \binom{n}{3}$ because a complete $3$-graph on fewer than $5$ vertices is obviously $\textup{Berge-}K_5$-free. Combining these results with the results of Gy\'arfas \cite{gyarfas} (concerning $\textup{Berge-}K_4$) mentioned earlier, we have the exact value of $\ex_3(n, \textup{Berge-}K_k)$ for all $n$ and $k>3$, as summarized below:

\begin{corollary} Let $n\ge 1$ and $k\ge 4$ be integers. Then
\[ \ex_3(n,\textup{Berge-}K_k) = \begin{cases} 
     |T_3(n,k-1)| & \textup{if } k\ge 6  \textup{ or } k=5, \, n\ge 6 \textup{ or } k=4,\, n\ge 6 \\[1em]
      \binom{n}{3} & \textup{if } k=5, n\le 4 \textup{ or } k=4, n\le 4 \\[1em]
       5 & \textup{if } k=4,\,n=5 \\[1em]
       9 & \textup{if } k=5,\,n=5.
   \end{cases}
\]
\end{corollary}

\end{document}